\documentclass[a4paper,french,12pt]{amsart}
\usepackage{bourbaki}
\usepackage[all]{xy}
\usepackage[frenchb]{babel}
\usepackage[applemac]{inputenc}
\usepackage{eucal}
\usepackage{mathrsfs}
\usepackage{amssymb}
\usepackage{amsxtra}
\usepackage{enumerate}
\makeatletter
\let\@@seccntformat\@seccntformat
\renewcommand*{\@seccntformat}[1]{%
  \expandafter\ifx\csname @seccntformat@#1\endcsname\relax
    \expandafter\@@seccntformat
  \else
    \expandafter
      \csname @seccntformat@#1\expandafter\endcsname
  \fi
    {#1}%
}
\newcommand*{\@seccntformat@section}[1]{%
  \S\csname the#1\endcsname\quad
}
\makeatother

\makeatletter
\renewcommand{\tocsection}[3]{%
  \indentlabel{\@ifnotempty{#2}{\S~\ignorespaces#1 #2.\quad}}#3}
\makeatother

\addto\captionsfrench{

}

\makeatletter
\@addtoreset{section}{part}
\@addtoreset{proposition}{section}
\@addtoreset{theoreme}{part}
\makeatother

\input cyracc.def 
 
\newtheorem*{theor}{Théorème}
\newtheorem{theo}{Théorème}
\newtheorem*{pro}{ Proposition}
\newtheorem*{lem}{ Lemme}
\newtheorem*{coro}{Corollaire}
\newtheorem{theoreme}{Théorème}[section]
\newtheorem{proposition}{ Proposition}[section]
\newtheorem{lemme}{ Lemme}[section]
\newtheorem{corollaire}{Corollaire}[section]

\newtheorem{definition}{Définition}[section]

\theoremstyle{remark}

\newtheorem{exemple}{\it Exemple}

\def \1{\mathbb {1}}
\def \RM{\mathbb {R}}
\def \NM{\mathbb{N}}
\def \ZM{\mathbb{Z}}
\def \CM{\mathbb{C}}

\def \Ker {{\rm Ker\,}}
\def \Im {{\rm Im\,}}

\def \Der {{\rm Der\,}}

\def \p {{\rm exp\,}}
\def \Id {{\rm Id\,}}
\def \id {{\rm id\,}}

\def \d{\partial}
\def\dt{\delta} 
\def\a{\alpha}

\def\l{\lambda}

\def\p{\varphi}

\def\G{\Gamma}   
\def\D{\Delta}
\def \s{\sigma}

\def \to{\longrightarrow} 

\def \alg{\mathfrak{g}}

\def \< {{\langle }}
\def \> {{\rangle }}
\def \( {\left( }
\def \) {\right) }

\newcommand{\Bt}{{\mathcal B}}

\newcommand{\Ft}{{\mathcal F}}
\newcommand{\Gt}{{\mathcal G}}

\newcommand{\Lt}{{\mathcal L}}
\newcommand{\Mt}{{\mathcal M}}

\newcommand{\Ot}{{\mathcal O}}

\newcommand{\Rt}{{\mathcal R}}

\newcommand{\Vt}{{\mathcal V}}

\renewcommand{\mod}{{\rm  mod\,}}
\title[D\'etermination finie]{D\'etermination finie sur un espace de Stein}
\author{ Mauricio  Garay}
\address{Institut für Mathematik\\
FB 08 - Physik, Mathematik und Informatik\\
Johannes Gutenberg-Universität Mainz\\
Staudinger Weg 9\\
55128 Mainz}
 
\begin{document}
\maketitle
\begin{abstract} On généralise le théorème de détermination finie sur les singularités isolés au cas des espaces analytiques de Stein.
\end{abstract} 
\section*{Introduction}
 La démonstration classique du théorème de détermination finie, pour les singularités isolées, utilise fortement les propriétés de l'algèbre locale. Elle ne s'adapte donc pas de manière directe au cas des variétés de Stein, à moins d'hypothèses drastiques. Le but de cet article est d'étendre et de généraliser ce théorème au cas d'une fonction sur une variété de Stein pour un idéal quelconque.

Avant dénoncer ce théorème, commençons par rappeler l'énoncé du théorème classique dans un cadre holomorphe. Notons $\Ot_n$ l'algèbre  des germes de fonctions analytiques à l'origine dans $\CM^n$.  L'{\em idéal jacobien du germe $f$}, noté $Jf$, est l'idéal de $\Ot_n$ engendré par les dérivées partielles de $f$. Il revient au même de dire que l'origine est un point critique isolé de $f$ ou bien que l'espace vectoriel
 $\Ot_n/Jf$ est de dimension finie. Le nombre 
 $$ \mu(f):=\dim_{\CM} \Ot_n/Jf$$
 est alors appelé le {\em nombre de Milnor} de $f$. 
 
L'anneau $\Ot_n$ est local et on note $\Mt_n$ l'idéal maximal des germes qui s'annulent en l'origine. La puissance $k$-ième de cet idéal maximal est égale aux germes dont le développement en série de Taylor à l'origine s'annule à l'ordre $k$.
 
 \begin{theor}[\cite{Mather_fdet,Tougeron,Tyurina}]  Pour tout germe de fonction holomorphe $f \in \Ot_n$ et tout germe 
 $g \in \Mt^{\mu(f)+2}_n$,  il existe un germe d'application biholomorphe
 $$\p:(\CM^n,0) \to ( \CM^n,0) $$
 tel que $f \circ \p=f+g$.
  \end{theor}
Par exemple, pour une fonction $f$ avec un point critique non-dégénéré à l'origine, on a $\mu(f)=1$. Le germe $f$ se ramène alors à un polynôme de degré $2$, par un changement de variables holomorphe. C'est le lemme de Morse complexe.

Considérons à présent le cas plus général qui fait l'objet de cet article. Soit $K \subset X$ un compact d'un espace analytique complexe $X$.
Pour tout voisinage $U,U'$ de $K$ avec $U \subset U'$, la restriction induit un morphisme d'algèbres
$$\G(U',\Ot_X) \to \G(U,\Ot_X) $$
où $\G(-,-)$ désigne le foncteur des sections globales. On a ainsi un système direct dont on note $\Ot_{X,K}$ la limite. Un élément de $\Ot_{X,K}$ est une fonction holomorphe définie sur un voisinage de $K$ dans $X$, mais on ne précise pas quel est ce voisinage. On note $\Mt_K \subset \Ot_{X,K}$ l'idéal des fonctions identiquement nulles sur $K$. 

 On dit que $K$ est un {\em compact de  Stein} s'il admet un système fondamental de voisinages de Stein.
Pour tout idéal $I  \subset \Ot_{X,K}$, on a une filtration
$$\Ot_{X,K} \supset I \supset I^2 \supset \dots .$$
Pour tout élément $f$ de $\Ot_{X,K}$, on note $I(f)$ l'image de l'application
$$\Mt^2_K\otimes_{\Ot_{X,K}} \Der_{\Ot_{X,K}}(I) \to \Ot_{X,K},\ a \otimes v \mapsto av(f). $$

 \begin{theo} 
 \label{T::det}
 Soit $K$ un compact de Stein d'un espace analytique complexe $X$ et $f \in \Ot_{X,K}$ un germe de fonction holomorphe. Supposons qu'il existe $\nu \in \NM$ tel que $I^{\nu}$ soit contenu dans l'idéal $I(f)$. Pour tout germe 
 $g \in I^{\nu}$,  il existe un germe d'application biholomorphe
 $$\p:(X,K) \to (X,K) $$
 tel que $f \circ \p=f+g$. 
  \end{theo}
 La démonstration que je donnerai repose fortement sur les propriétés des variétés de Stein. Par conséquent, je ne sais pas si, à l'instar du résultat classique, ce théorème reste vrai dans un cadre $C^\infty$.  En revanche, on a une variante analytique réelle immédiate de ce résultat.

On définit l'idéal jacobien de $f$ comme l'image de l'application
$$\Der_{\Ot_{X,K}}(\Ot_{X,K}) \to \Ot_{X,K},\ v \mapsto v(f). $$
Pour l'idéal $\Mt^k$, on peut améliorer la borne $\nu$ donnée par le théorème précédent~:
 \begin{theo}
 \label{T::det_bis}Soit $K$ un compact de Stein d'un espace analytique complexe $X$ et $f \in \Ot_{X,K}$ un germe de fonction holomorphe. Supposons qu'il existe $\mu \in \NM$ tel que $\Mt^\mu_K$ soit contenu dans l'idéal jacobien de $f$. Pour tout germe 
 $g \in \Mt^{\mu+2}_K$,  il existe un germe d'application biholomorphe
 $$\p:(X,K) \to (X,K) $$
 tel que $f \circ \p=f+g$.
  \end{theo}
  
 Pour tout ideal $I$ de codimension $k$ dans un anneau local $A$, la puissance $k$-ième de l'idéal maximal est contenue dans $I$. Ainsi, dans le cas où $X=\CM^n$ et $K$ est un point, on retrouve le théorème de détermination finie classique.

\`A titre d'exemple, prenons $X=\left( \CM/2\pi \ZM \right) \times \CM=\{ (\theta,r) \}$ et soit  $K$  le cercle réel~:
$$K =\{ (\theta,r) \in \left(\CM/2\pi \ZM \right) \times \CM : \theta \in \RM/2\pi \ZM,\ r=0 \}.$$  Les éléments de $\Ot_{X,K}$ sont des séries de la forme $\sum_{m,n} a_{m,n}r^m e^{i\, n\theta}$. Prenons 
$$f:(r,\theta) \mapsto r^k.$$
Le théorème affirme alors que toute série de la forme
$r^k+r^{k+1}g(r,\theta) $ se ramène à $r^k$ par un changement de variables biholomorphe. Ce qui se vérifie directement en utilisant le changement de variables
$$ \p:(\theta,r) \mapsto (\theta, r \sqrt[k]{1+rg(r,\theta)}).$$

Chacun des théorèmes précédents peut-être reformulé en termes d'actions de groupes, par exemple~:
\begin{theo} Soit $I$ un idéal de $\Ot_{X,K}$, $f \in \Ot_{X,K}$ et $\nu$ tel que $I^{\nu} \subset I(f)$. L'espace  affine $f+I^{\nu}$ est contenu dans l'orbite de $f$ sous l'action du groupe des automorphismes de l'algèbre $\Ot_{X,K}$.
\end{theo}

 Le module des dérivations de l'algèbre $\Ot_{X,K}$ constitue l'analogue, en dimension infinie,
 de l'algèbre de Lie de son groupe d'automorphismes. On semble donc reconnaître un théorème qui relie l'action d'un groupe avec celle de sa linéarisation. Ce point de vue heuristique est développé rigoureusement dans cet article, prolongeant ainsi le résultat de \cite{groupes}.


 \section{La catégorie des espaces vectoriels échelonnés}
\subsection{Echelonnement d'un espace vectoriel topologique}
\label{SS::definition}

Une {\em $S$-échelle de Banach} est une famille décroissante d'espaces de Banach $(E_s)$, $s \in ]0,S[$, telle que 
les inclusions 
$$E_{s+\s} \subset E_s,\ s \in ]0,S[,\ \s \in ]0,S-s[$$ 
soient de norme au plus~$ 1$.

Soit $E$ un espace vectoriel topologique. Un {\em $S$-échelonnement} de $E$ est une échelle  $(E_s)$
de sous-espaces de Banach de $E$ telle que
\begin{enumerate}[{\rm i)}] 
\item  $\displaystyle{E=\lim_{\to} E_s=\bigcup_{s \in ]0,S[} E_s}$ ;
\item la topologie  limite directe de la topologie des espaces de Banach $E_s$ coïncide avec celle de $E$.
\end{enumerate}
 L'intervalle $]0,S[$ s'appelle {\em l'intervalle d'échelonnement}.  Si $F$ est un sous-espace vectoriel fermé d'un espace vectoriel échelonné $E$ alors $E/F$ est échelonné par les espaces de Banach
$E_s/(E \cap F)_s$.
  
La notion d'échelonnement vise à transférer les propriétés de $E$ aux espaces de Banach $E_s$, mais l'objet que l'on étudie reste $E$ et non pas l'échelle de Banach. Lorsque le paramètre $S$ ne joue pas de rôle particulier, nous parlerons simplement d'échelle de Banach ou d'espace vectoriel échelonné. 

Pour un ensemble $A \subset \CM^n$, nous noterons $ \mathring{A}$ son intérieur. Une suite croissante de compacts $K=(K_s),\  s \in [0,S]$ est appelé {\em une famille exhaustive  de compacts}  si $K_s$ est contenu l'intérieur de $K_{s'}$ pour tout $s',s$ avec $s'>s$.

\begin{exemple}  Soit $(K_s), \ s \in [0,S]$ une  famille exhaustive de compacts de $\CM^n$.  Les espaces vectoriels
$$E_s:=C^0(K_s,\CM) \cap \G(\mathring{K_s},\Ot_{\CM^n,0} )$$ 
sont des espaces de Banach pour la norme
$$| f |_s:=\sup_{z \in K_s} |f(z)|. $$
Ils définissent un échelonnement de l'espace vectoriel topologique $\Ot_{\CM^n,K}$. 
Les suites $(E_{s^2})$, $(E_{3s})$ donnent d'autres exemples d'échelonnement de $\Ot_{\CM^n,K}$, avec les mêmes espaces de Banach~(voir section \ref{S::analytique} pour plus de détails).
\end{exemple}

L'utilisation d'échelles de Banach en analyse remonte aux fondements de l'analyse fonctionnelle. On la trouve par exemple dans la démonstration du théorème de Cauchy-Kovalevskaïa donnée en 1942 par Nagumo~\cite{Nagumo} (voir également \cite{Ovsyannikov}).
 Elle est également à la base de la démonstration proposée par Kolmogorov du théorème des tores invariants~\cite{Kolmogorov_KAM}. 
 
 Cependant ces auteurs ne considèrent qu'une échelle fixe, l'idée de considérer toutes les échelles possibles d'un sous-espace vectoriel topologique est déjà présente dans la thèse de Grothendieck~\cite{Grothendieck_PTT}. En revanche, Grothendieck n'utilise pas le choix d'un paramétrage de l'échelle comme une donnée supplémentaire.
 \subsection{Filtration d'un espace vectoriel échelonné}
 Soit $E$ un espace vectoriel échelonné. Les sous-espaces vectoriels
 $$E^{(k)}=\{ x \in E: \exists C,\tau,\ | x|_s \leq Cs^k,\ \forall s \leq \tau \}$$
 filtrent l'espace $E$~:
 $$E:=E^{(0)} \supset E^{(1)}\supset E^{(2)} \supset \cdots. $$
 \begin{exemple} Considérons les polycylindres
 $$K_s=\{ (z_1,z_2,\dots,z_n) \in \CM^n: | z_1| \leq s\, ,\dots,\ | z_n | \leq s \}.$$ 
 Comme précédemment, échelonnons l'espace vectoriel 
 $$\Ot_n:=\Ot_{\CM^n,K}, K=\{ 0 \} \subset \CM^n$$ par les espaces de Banach
 $$E_s:=C^0(K_s,\CM) \cap \G(\mathring{K_s},\Ot_{\CM^n,0} ).$$
   L'anneau $\Ot_{\CM^n,K}$ est local d'idéal maximal
 $$\Mt_n:=\{ f \in \Ot_{\CM^n,K}: f(0)=0 \}. $$ La filtration d'espace vectoriel échelonné coïncide avec celle donnée par les puissances de l'idéal maximal~:
 $$(\Ot_n)^{(k)}=\Mt_n^k. $$
 \end{exemple} 

\begin{definition}Soit $E$ un espace vectoriel échelonné. L'ordre d'un vecteur $x \in E$ est le plus grand $k \geq 0$ tel que $x \in E^{(k)}$.
\end{definition}
\subsection{Morphismes d'un espace vectoriel échelonné}
\label{SS::morphismes}
Soit $E,F$ deux espaces vectoriels $S$-échelonnés.

Nous dirons d'une application linéaire que c'est  un {\em morphisme} entre des espaces vectoriels échelonnés $E,F$, si pour tout $s' \in ]0,S[$, 
il existe $s \in ]0,S[$ tel que l'espace de Banach $E_{s'}$ est envoyé continûment dans $F_s$. Nous avons ainsi définit la  {\em catégorie des espaces vectoriels échelonnés}.

Nous désignerons par $\Lt(E,F)$ l'espace vectoriel des morphismes de $E$ dans $F$ et lorsque $E=F$, nous utiliserons la notation $\Lt(E)$ au lieu de $\Lt(E,E)$. Il n'y pas de raison, a priori, pour que $\Lt(E,F)$ coïncide avec l'espace des applications linéaires continues de $E$ dans $F$, mais dans les exemples concrets que nous allons traiter ce sera toujours le cas.
 
   Si $\| \cdot \|$ désigne la norme d'opérateur sur l'espace de Banach $\Lt(E_{s'},F_s)$, nous noterons $\| u \|$ la norme de l'opérateur défini par restriction de $u$ à $E_{s'}$.

Le noyau d'un  morphisme $ u:E \to F$ entre espaces vectoriels $S$-échelonnés est un espace vectoriel $S$-échelonné par~:
$$ (\Ker u)_s=E_s \cap \Ker u,\ s \leq S.$$
 
Venons-en à la notion de convergence d'une suite de morphismes. La norme d'opérateur induit sur  les espaces vectoriels $\Lt(E_{s'},F_s)$, une structure d'espace de Banach. 
\begin{definition}Une suite de morphismes $(u_n)$ de $\Lt(E,F)$ converge vers un morphisme $u \in \Lt(E,F)$ si pour tout $s' \in ]0,S[$, il existe $s \in ]0,S[$ tel que la restriction de $(u_n)$ définisse une suite de $\Lt(E_{s'},F_s)$ qui converge vers la restriction de $u$.
\end{definition}
Un sous-ensemble $X$ de $\Lt(E,F)$ sera dit {\em fermé} si toute suite convergente de points de $X$ à sa limite dans $X$. (L'utilisation du mot «fermé» est légèrement abusive, car il ne s'agit pas a priori du complémentaire d'un ouvert.)
\begin{exemple}  Comme précédemment, échelonnons l'espace vectoriel 
 $$\Ot_n:=\Ot_{\CM^n,K}, K=\{ 0 \} \subset \CM^n$$ par les espaces de Banach
 $$E_s:=C^0(K_s,\CM) \cap \G(\mathring{K_s},\Ot_{\CM^n,0} ).$$
Fixons $\l>0$, l'application
$$u:\Ot_n \to \Ot_n,\ f \mapsto [z \mapsto f(\frac{z}{\l})] $$
est un morphisme de $L(\Ot_n)$ car
$$f \in C^0(K_s,\CM) \cap \G(\mathring{K}_s,\CM) \implies  u(f) \in C^0(K_{\l s},\CM) \cap \G(\mathring{K}_{\l\, s},\CM)  .$$ Plus généralement, on vérifie, sans difficultés, que  l'espace vectoriel $L(\Ot_n)$ coïncide avec celui des applications linéaires continues pour la topologie forte. 
\end{exemple}
 \subsection{  Morphismes born\'es}

\begin{definition} Un morphisme $ u \in \Lt(E,F)$ entre deux espace vectoriel échelonnés  est appelé  un $\tau$-morphisme
si pour tout $s'  \in ]0,\tau]$ et pour tout $s \in ]0,s'[$, on a l'inclusion $ u(E_{s'}) \subset F_s$ et $ u$ induit par restriction une application linéaire continue
$$u_{s',s}~:~E_{s'}~\to~F_s.$$
\end{definition}
On a alors des diagrammes commutatifs
$$\xymatrix{ & \ F_s \ar@{^{(}->}[d]\\
 E_{s'} \ar[r]^-{u_{\mid E_{s'}}} \ar[ru]^{u_{s',s}}  & F }
$$
pour tout $s'  \in ]0,\tau]$ et pour tout $s \in ]0,s'[$, la flèche verticale étant donnée par l'inclusion $F_s \subset F$.
\begin{exemple}  L'application $u$ construite dans l'exemple du n° précédent n'est pas un $\tau$-morphisme alors que tout opérateur
diffé\-ren\-tiel définit un $\tau$-morphisme.
\end{exemple}
\begin{definition}
\label{D::borne}
  Un $\tau$-morphisme $ u:E \to F$ d'espaces vectoriels $S$-échelonnés est dit 
    $k$-borné, $k \geq 0 $ s'il existe un réel $C>0$ tel que~:
  $$| u(x) |_s \leq C \sigma^{-k} | x |_{s+\sigma},\ {\rm pour\ tous\ }\ s \in ]0,\tau[,\ \s \in ]0,\tau-s],\ x\in E_{s+\s} . $$
\end{definition}
Un morphisme est dit {\em $k$-borné} (resp. {\em borné}) s'il existe $\tau$ (resp. $\tau$ et $k$) pour lequel (resp. lesquels) c'est un $\tau$-morphisme $k$-borné. Lorsque $E=E_s$ et $F=F_s$ sont des espaces de Banach, on retrouve la définition habituelle de morphismes bornés.
(Nous n'utiliserons pas la notion plus générale d'application linéaire bornée d'un espace localement convexe, notre terminologie ne devrait donc pas porter à confusion.)  
L'espace vectoriel des $\tau$-morphismes (resp. des morphismes)  $k$-bornés
entre $E$ et $F$ sera noté $\Bt^k_\tau(E,F)$ (resp. $\Bt^k(E,F)$). On note $N_\tau^k(u)$ la plus petite constante~$C$ vérifiant l'inégalité de la définition \ref{D::borne}.   
\begin{exemple} Considérons,  l'échelonnement de $\Ot_n$ définit à l'aide des polycylindres $K_s$. D'après les inégalités de Cauchy, tout opérateur diffé\-rentiel d'ordre $k$ est $k$-borné.
\end{exemple}

\begin{proposition}
\label{P::quotient}
Considérons un diagramme exact d'espaces vectoriels échelonnés
$$\xymatrix{ F_1 \ar[r] \ar[d]&E_1 \ar[r] \ar[d]^u& E_1/F_1 \ar[r] \ar[d]^v & 0\\
F_2 \ar[r] &E_2 \ar[r] & E_2/F_2 \ar[r]  & 0 }$$
Si $u$ est un $\tau$-morphisme $k$-borné alors $v$  est également un $\tau$-morphisme $k$-borné et de plus
$$N^k_\tau(v) \leq N^k_\tau(u). $$ 
\end{proposition}
La démonstration est immédiate.

\subsection{L'échelle de Banach $(\Bt^k_\tau(E,F))$.}
\begin{proposition} Si $E,F$ sont des espaces vectoriels $S$-échelonnés alors les espaces vectoriels normés
$(\Bt^k_\tau(E,F),N_\tau^k),\ \tau \in ]0,S[$, forment une $S$-échelle de Banach.  
\end{proposition}
\begin{proof}
  La seule difficulté consiste à montrer que l'espace vectoriel $\Bt^k_\tau(E,F)$ est complet pour la norme $N_\tau^k$, pour tout $\tau \in]0,S[$.
  
 Je dis que toute suite de Cauchy  $(u_n) \subset \Bt^k_\tau(E,F)$  converge vers un morphisme $u \in \Lt(E,F)$ au sens de \ref{SS::morphismes}.
   Soit donc $s' \in ]0, \tau[$ et $s \in ]0,s'[$. Comme les $(u_n)$ sont des $\tau$-morphismes, ils induisent, par
 restriction, des applications linéaires continues 
 $$v_n:E_{s'} \to F_s.$$  
 Par définition de la norme $N_\tau^k$, la suite $(v_n)$ est de Cauchy dans l'espace de Banach $\Lt(E_{s'},F_s)$ donc convergente. Ceci démontre l'affirmation.  
   
 Montrons à présent que si une suite de Cauchy  $(u_n) \subset \Bt^k_\tau(E,F)$  converge vers un morphisme $u \in \Lt(E,F)$ alors
 $u$ est dans $ \Bt^k_\tau(E,F)$. L'inégalité
  $$| N_\tau^k( u_n) - N_\tau^k( u_m) | \leq  N_\tau^k( u_n- u_m) $$
   montre que  la suite $(N_\tau^k(u_n))$ est de Cauchy dans $\RM$ donc majorée par un constante~$C>0$.
   On a alors les inégalités~:
   $$| u(x) |_s \leq | u(x)-u_n(x) |_s+C\s^{-k}| x |_{s+\s}, \ {\rm\ pour\ tout\ } n, $$ 
  pour tout $x \in E_s$, pour tout $s \in ]0,\tau]$ et pour tout $\s \in ]0,\tau-s]$. Par conséquent, le $\tau$-morphisme limite $ u$ est $k$-borné de norme au plus égale à $C$.  La proposition est démontrée.
\end{proof}  
La suite $(\Bt^k_\tau(E,F),N_\tau^k),\ \tau \in ]0,S[$ munit l'espace vectoriel $\Bt^k(E,F)$ d'une structure d'espace vectoriel échelonné. Ainsi, l'échelonnement des espaces vectoriels $E,F$ se propage à celui des espaces de morphismes bornés.


\section{L'application exponentielle}
\subsection{Convergence de la série exponentielle}
\label{SS::produits}
 \begin{proposition}
 Soit $u$ un $\tau$-morphisme $1$-borné d'un espace vectoriel échelonné $E$. Si l'inégalité $3N_s^1(u) ~<~s $ est satisfaite pour tout $s \leq \tau$ alors
 la série 
 $$e^u:=\sum_{j \geq 0}\frac{u^j}{j!} $$
 converge vers un morphisme de $E$, et plus précisément
$$| e^u x |_{\l s} \leq \sum_{j \geq 0} \frac{(3N^1_s(u))^j }{(1-\l)^j s^j} | x |_s=\frac{1}{1-\frac{3N^1_s(u) }{(1-\l) s} }   |x |_s$$
pour tous $\l \in ]0, 1-\frac{3N^1_s(u)}s[$, $s \in ]0, \tau]$ et $x \in E_s$.  
\end{proposition}
\begin{proof}
Si $u,v$ sont des morphismes,  respectivement $k$ et $k'$ borné, alors leur composition $u  v$ est
  $(k+k')$-borné et on a l'inégalité
  $$N_\tau^{k+k'}( u v) \leq 2^{k+k'} N_\tau^k(u)N_\tau^{k'}(v) .$$
  En effet~:
 $$| (u v)(x) |_s \leq N_\tau^k(u) \frac{2^k}{\sigma^{k}} |v (x)|_{s+\sigma/2} \leq   N_\tau^k(u)N_\tau^{k'}(v) \frac{2^{k+k'}}{\sigma^{k+k'}} |x|_{s+\sigma} $$
 pour  tout $x\in E_{s+\s}$. Plus généralement~:
\begin{lemme}
\label{L::morphismes}
Le produit de $n$ morphismes $k_i$ bornés $u_i,\ i=1,\dots,n$, est un morphisme $k $-borné avec $k:=\sum_{i=1}^n k_i$ et
  $$N_\tau^k( u_1  \cdots  u_n) \leq n^k \prod_{i=1}^n N_\tau^{k_i}(u_i)\, .$$
  De plus si tous les $u_i$ sont égaux à un morphisme $1$-borné $u$, on a~:
     $$\frac{N_\tau^n( u^n )}{n!} \leq 3^n  N_\tau^1(u)^n. $$
 \end{lemme} 
 \begin{proof}
La première partie du lemme s'obtient en découpant l'intervalle $[s,s+\s]$ en $n$ parties égales, comme nous l'avons fait précé\-demment pour $n=2$.
 Prenons tous les $u_i$ égaux et $1$-bornés. D'après le lemme, on a alors
  $$N_\tau^n( u^n ) \leq n^n  N_\tau^1(u)^n.$$ 
 Une variante de la formule de Stirling montre que
  $$ n^n  \leq 3^n n!\,.$$
En effet, en utilisant l'expression intégrale suivante de  la fonction $\G$~:
  $$n!=\G(n+1) =n^{n+1} \int_{0}^{+\infty} e^{n (\log t-t)} dt .$$
 et l'estimation
  $$  -\frac{1}{2}(t-1)^2-1 \leq  \log t-t,$$
  il vient~:
  $$\G(n+1) \geq  n^{n+1}e^{-n} \int_{0}^{+\infty} e^{-\frac{1}{2}(t-1)^2} dt \geq  n^{n+1}e^{-n} \geq n^n 3^{-n}. $$
 Ceci démontre que
   $$\frac{N_\tau^n( u^n )}{n!} \leq 3^n  N_\tau^1(u)^n. $$
   \end{proof}
Nous pouvons à présent conclure la démonstration de la proposition. On a 
$$| e^u x |_{\l s} \leq \sum_{j \geq 0} \frac{1}{j!}| u^j x |_{\l s}. $$
Le morphisme $u^j$ est $j$-borné et on a l'inégalité~:
$$ \frac{1}{j!} | u^j x |_{\l s} \leq \frac{(3N^1_s(u))^j }{(1-\l)^j s^j} | x |_s, $$
ce qui démontre la proposition.
\end{proof}

Finalement, remarquons que deux morphismes 1-bornés $u,v \in \Bt^1(E)$  qui commutent et qui satisfont aux conditions de la proposition précédente vérifient l'égalité
 $$e^{u+v}=e^{u}e^{v}.$$ En effet, si $u$ et $v$ commutent alors les suites
 $$A_n:=\sum_{j= 0}^n \frac{u^j}{j!},\ B_n=\sum_{j= 0}^n \frac{v^j}{j!}. $$
 vérifient
 $$A_nB_n=\sum_{j= 0}^n \frac{(u+v)^j}{j!} $$
et si des suites de morphismes $(A_n), (B_n)$ convergent respectivement vers $A,B$ alors $(A_nB_n)$ converge vers $AB$. Le cas particulier $v=-u$ montre que l'exponentielle d'un morphisme $1$-borné est inversible.

\begin{exemple} Considérons l'espace vectoriel $\Ot_{\CM,0}$ échelonné comme pré\-cé\-demment.
 L'exponentielle de $\l z \d_z$ converge et donne l'automorphisme
 d'algèbre
 $$z \mapsto (\sum_{n \geq 0} \frac{(\l z\d_z)^n}{n!} )z=e^\l z.  $$
 Plus généralement, toute dérivation de la forme 
$$zh(z)\d_z,\ h\in \Ot_{\CM,0} $$
est exponentiable. En résumé, l'algèbre de Lie $\alg$ des dérivations de $\Ot_{\CM,0}$ est filtrée
$$\alg \supset \alg^{(1)} \supset \alg^{(2)} \supset \cdots  $$
avec
$$ \alg^{(k)}=\{ z^kh(z)\d_z,\ h \in \Ot_{\CM,0} \}$$
et tout élément de $\alg^{(1)}$ est exponentiable.

Notons $\Mt_{\CM,0}$, l'idéal maximal de l'anneau local $\Ot_{\CM,0}$~:
$$\Mt_{\CM,0}=\{ f \in \Ot_{\CM,0}: f(0)=0 \}.$$
 Lorsque $v \in \alg^{(2)}$ et $f \in\Mt_{\CM,0}^k$,on a~:
 $$ e^v f=f+v \cdot f\ (\mod \Mt_{\CM,0}^{k+1}). $$
 Ce qui donne un sens précis au fait que l'action infinitésimale de $e^v$ est donnée par la dérivation le long de $v$. En général, ce n'est plus vrai si l'on fait seulement l'hypothèse $v \in \alg^{(1)} $. Par exemple pour $\displaystyle{v=-\frac{z}{2}\d_z}$ et $f=z^2$, on trouve~:
 $$e^v f=z^2-z^2+\frac{z^2}{2!}+\dots+(-1)^n\frac{z^2}{n!}+\dots=e^{-1} z^2$$
alors que $f+v ( f)=0$.   
\end{exemple}
\subsection{Théorème principal}
  \label{SS::exp}
\begin{theoreme} 
\label{T::produits}
Soit $E$ un espace vectoriel échelonné.   Soit $(u_n) \subset \Bt^1_\tau(E)$ une suite de $\tau$-morphismes $1$-bornés.
Si l'inégalité
$$3\sum_{i \geq 0} N_s^1(u_i) < s  $$
est vérifiée pour tout $s \leq \tau$ alors la suite $(g_n)$ définie par
$$g_n:=e^{u_n}e^{u_{n-1}}\cdots e^{u_0}  $$
converge vers un élément inversible de $\Lt(E)$.
 \end{theoreme}
 \begin{proof}
Commençons par le
 \begin{lemme} 
\label{L::produits}
Soit $(u_n)$ une suite de $\tau$-morphismes $1$-bornés exponentiables.
Pour tout $s \leq \tau$ et tout $x \in E_s$, on a l'inégalité
$$| g_n x |_{\l s} \leq  \frac{1}{1-\frac{3}{(1-\l)s}\sum_{i=0}^nN_s^1(u_i)} | x |_s $$
pourvu que 
$$\sum_{i=0}^n N_s^1(u_i) < \frac{(1-\l)s}{3} . $$
 \end{lemme}
\begin{proof}
Notons $C_{j,n} \subset \ZM^j$ l'ensemble des éléments $ i=(i_1,\dots,i_j)$ dont les coordonnées sont dans
 $\{0,\dots,n \}$.  On a alors la formule
 $$ \frac{1}{1-\a(\sum_{k=0}^n z_k)}=\sum_{j \geq 0} \a^j \sum_{i \in C_{j,n}}  z_i,\ z_i:=z_{i_1}z_{i_2}\cdots z_{i_j}$$
 pour tout $\a \in \RM$.

Pour tout $i=(i_1,\dots,i_j) \in C_{j,n}$ on note $\s(i)$ le vecteur dont les composantes sont obtenues à partir de $i$ par permutation pour que
$\s(i)_p \geq \s(i)_{p+1}$. 

On pose
$$u[i]:= u_{\s(i)_1} u_{\s(i)_2} \cdots u_{\s(i)_j},\ i \in C_{j,n}. $$
Développons $g_n$ en série puis regroupons les termes de la façon suivante~:
$$g_n=\sum_{j \geq 0}\frac{1}{j!} (\sum_{i \in C_{j,n}}  u[i]) =1+\sum_{i=0}^n u_i+\frac{1}{2}(\sum_{i=0}^n u_i^2
+\sum_{j=0}^n \sum_{i=j+1}^n 2 u_i u_j) +\dots .$$

Posons 
$$z_{i,s}:=N_s^1(u_{i_1})N_s^1(u_{i_2})\cdots N_s^1(u_{i_n}).$$
D'après le lemme  \ref{L::morphismes}, on a l'inégalité~:
$$\frac{1}{j!}N_s^j( u[i])  \leq 3^j\prod_{p=0}^jN_s^1(u_{i_p})= 3^j z_{i,s}$$
et par suite
$$\left| u[i] (x)\right|_{\l s} \leq   \left(\frac{ 3}{(1-\l)s} \right)^j z_{i,s}  | x|_s,\ \forall \l \in ]0,1[. $$
Posons
$$ \a=\frac{3}{(1-\l)s},$$
On obtient ainsi l'estimation
$$| g_n x |_{\l s} \leq  \left( \sum_{j \geq 0}\a^j \sum_{i \in C_{j,n}}  z_{i,s} \right)| x|_s=\frac{1}{1-\a(\sum_{k=0}^n z_{k,s})}| x|_s  .  $$
Ceci démontre le lemme.
 \end{proof}
 
 Achevons la démonstration du théorème. Pour cela, fixons $s~\in~]0,\tau]$.  
Par hypothèse, on a
$$3 \sum_{i \geq 0} N_s^1(u_i)<s .$$
On peut donc choisir $\l \in]0,1[$ tel que
$$3 \sum_{i \geq 0} N_s^1(u_i)<(1-\l)s.  $$

 Notons $\| \cdot \|_{\l}$ la norme d'opérateur dans $\Lt(E_s,E_{\l s})$. 
Le lemme précédent donne l'estimation
$$ \| g_n \|_\l \leq    \frac{1}{1-\frac{3}{(1-\l)s}\sum_{i \geq 0} N_s^1(u_i)}.$$
La suite $(g_n)$ définit donc, par restriction, une suite uniformément bornée d'opérateurs dans $\Lt(E_s,E_{\l s})$. 

Soit à présent $\mu \in ]0,1[$ tel que
 $$3\sup_{i \geq 0} N_{\l s}^1(u_i) <(1-\mu)\l s .$$ 
 Nous allons montrer que la suite $(g_n)$ définit, par restriction, une suite de Cauchy dans $\Lt(E_s,E_{\mu \l s})$. La proposition en découlera, car ce dernier est un espace de Banach pour la norme d'opérateur. 

Je dis que la série de terme gé\-né\-ral $ \| g_n-g_{n-1}\|_{\l \mu }$ est convergente. Pour le voir, écrivons
$$g_n-g_{n-1}=(e^{u_n}-\Id)g_{n-1} $$
où $\Id \in \Lt(E)$ désigne l'application identité.

En développant l'exponentielle en série, on obtient l'inégalité~:
$$ | (e^{ u_n}-\Id) y |_{\l \mu s} \leq  \left( \sum_{j \geq 0} \frac{ ( 3 N_{\l s}^1(u_n))^{j+1} }{((1-\mu) \l s )^{j+1} } \right) | y |_{\l s}=
\frac{3}{1-\mu-\frac{3N^1_{\l s}(u_n)}{\l s}} \frac{N^1_{\l s}(u_n)}{\l s} | y |_{\l s} ,$$
pour tout $y \in E_{\l s}$. En prenant $y=g_n x$, ceci nous donne l'estimation
$$ \| (e^{ u_n}-\Id) g_{n-1} \|_{\l \mu} \leq   \frac{3 C_\l}{1-\mu-\frac{3N^1_{\l s}(u_n)}{\l s}} \frac{N^1_{\l s}(u_n)}{\l s}.$$
 
La quantité 
$$K_{\l,\mu}:=\sup_{n \geq 0} \frac{3 C_\l}{1-\mu-\frac{3N^1_{\l s}(u_n)}{\l s}} $$
est finie car la suite $3N^1_{\l s}(u_n)$ tend vers $0$ lorsque $n$ tend vers l'infini. 
Nous avons donc montré l'estimation
$$\| g_n-g_{n-1} \|_{\l \mu} \leq K_{\l,\mu} \frac{N^1_{\l s}(u_n)}{\l s}.$$
Il ne nous reste plus qu'à utiliser l'inégalité triangulaire pour voir que $(g_n)$ définit une suite de Cauchy de l'espace de Banach $\Lt(E_s,E_{\mu \l s})$~:
$$\| g_{n+p}-g_{n}  \|_{\l \mu } \leq \sum_{i = 1}^{p}  \| g_{n+i}-g_{n+i-1}\|_{\l \mu  } \leq K_{\l,\mu}\left(\sum_{i = 1}^{p} \frac{3N^1_{\l s}(u_{n+i})}{\l s}\right).$$
Ceci montre bien que la suite $(g_n)$ converge vers un élément $g \in \Lt(E)$. On démontre de même que la suite $(h_n)$  définie par
$$h_n=e^{- u_0}e^{- u_1} \cdots e^{-u_n}$$ 
converge vers un élément $h \in \Lt(E)$. Pour tout $n \in \NM$, on a~:
$$g_nh_n=h_ng_n=\Id$$
 donc $gh=hg=\Id$. Ce qui montre que $h$ est l'inverse de $g$.
 Le théorème est démontré. 
 \end{proof}
 \section{D\'etermination finie sur un espace vectoriel \'echelonn\'e}
 \subsection{\'Enoncé du théorème et principe de la démonstration}
 \begin{theoreme}
 \label{T::groupes} Soit $E$ un espace vectoriel échelonné, $a \in E$, $M$ un sous-espace vectoriel fermé de $E$, $\alg$ un sous-espace vectoriel de $\Bt^1(E)^{(2)}$ qui préserve $M$, $G$ un sous-groupe fermé de $\Lt(E)$ contenant $\exp(\alg)$. Si
 l'application  
 $$\rho:\alg \to M,\ u \mapsto u \cdot a$$
  possède un inverse à droite borné alors l'orbite de $a$ sous l'action de $G$ est égale à $a+M$.
 \end{theoreme}

 Notons 
 $$j:E \mapsto \alg$$
   l'inverse de l'application $\rho$. Soit $b \in M$, on cherche $g \in G$ tel que  $g \cdot a=a+b$.
Pour cela, on  considère les suites $(b_n)$ et $(u_n)$ définies par~:
\begin{enumerate}[1)]
\item $b_{n+1}:=e^{-u_n}(a+b_n)-a$ ;
\item  $u_{n+1}:=j(b_{n+1}).$
\end{enumerate}
avec $ b_0=b,\ u_0=j(b) $.

Dans cette itération, la suite $(u_n)$ peut également être définie par la formule
$$u_{n+1}=j((e^{-u_n}(a+u_n\, a )-a)) $$
 
 Supposons que la suite formée par les produits
$$g_n:= e^{u_n} \dots e^{u_1}  e^{u_0} $$
converge vers une limite $g$ et que $(u_n)$ tende vers $0_{\alg}$. Dans ce cas, la suite $(x_n)=(g_n x)$ converge vers $a$. En effet, par définition de $j$, on a
$$ u_n(a)=b_n$$
et en passant à la limite sur $n$, dans les deux membres de l'égalité, on trouve 
$$0_E=b'.$$ 
En passant, maintenant à la limite dans l'égalité
$$g_n(a+b)=a+b_n$$
on trouve bien $g(a+b)=a$. CQFD.
 
Le théorème sera donc démontré pourvu que $(g_n)$ soit convergente et que $(u_n)$ tende vers $0_{\alg}$.  D'après le théorème \ref{T::produits},  il suffit pour cela de  montrer que la série
$ \sum_{n \geq 0} N^1_s(u_n)$ est majorée par $3s$ pour tout $s$ suffisamment petit, et comme nous allons le voir c'est un fait presque immédiat.

\subsection{Démonstration du théorème \ref{T::groupes}}
\label{SS::homogene_dem}

Soit $]0,S[$ l'intervalle d'éche\-lonnement de $E$.
 Quitte à remplacer $S$ par $S'<S$, on peut supposer que $a \in E_S$ et $S<1/2$. Par ailleurs, quitte à multiplier toutes les normes par une même constante, on peut
 supposer que
 $$|a |_S \leq 1.  $$
 \begin{lemme}
 \label{L::reste}
 Pour tout $\tau$-morphisme $1$-borné $u$ vérifiant la condition 
 $$\frac{3N^1_\tau(u)}{\tau-s} \leq \frac{1}{2}, $$
  on a l'inégalité~:
 $$| (e^{-u}(\Id +u)-\Id) a|_s \leq  \frac{1}{(\tau-s)^2} N^1_\tau(u)^2,$$
 pour tout $s \in ]0,\tau[$.
\end{lemme}  
\begin{proof}
On a l'égalité~:
$$ e^{-u}(\Id +u)-\Id=\sum_{n \geq 0} \frac{(n+1)}{(n+2)!}(-1)^{n+1}u^{n+2} \ ;$$
d'où l'estimation~: 
$$ | \sum_{n \geq 0} (-1)^{n+1} \frac{(n+1)}{(n+2)!}u^{n+2} ( a) |_s \leq \sum_{n \geq 0}  
\frac{(n+1)3^{n+2}}{(\tau-s)^{n+2}}N^1_\tau(u)^{n+2} $$
car $|a|_S\leq 1.$ Comme
$$\frac{3N^1_\tau(u)}{\tau-s} \leq 1 $$
le membre de droite est égal à
$$x^2\sum_{n \geq 0}(n+1)x^n=\frac{x^2}{(1-x)^2},\ {\rm \ avec\ }x= \frac{3N^1_\tau(u)}{\tau-s}.$$
En utilisant l'inégalité
$$\frac{x^2}{(1-x)^2} \leq 1,\ \forall x \in [0,\frac{1}{2}] ,$$
on trouve bien la majoration du lemme.
 \end{proof}
En prenant $\tau=2s$ dans le lemme, on voit, en particulier, que l'ordre de $u_n$ augmente strictement avec $n$.  

Fixons $k \in \NM$ tel que $j$ soit $k$-borné. Quitte à réduire l'intervalle d'échelonnement, on peut supposer que $j$ est un $S$-morphisme $k$-borné. Puisque l'ordre de $u_n$ augmente strictement avec $n$, on peut supposer, quitte à remplacer la suite $u_n$ par $u_{n+k+1}$, que $u_0$ est d'ordre $k+3$.

Considérons les suites $(\s_n)$ et $(s_n)$ définies par  
$$\s_n=\frac{s}{2^{n+2}},\ s_{n+1}=s_n-2\s_n,\ s_0=2s .$$

La fonction $ N^1_{2s}(u_0)$ décroit avec $s$ au moins à la vitesse de $s^{k+3}$. Il existe donc $\tau>0$ tel que pour tout $s \leq \tau$, l'inégalité suivante soit vérifiée au rang $n=0$~:
$$(*)\ N^1_{s_n}(u_n) \leq m\s_n^{k+2} . $$
 avec $m:=\frac{1}{2^{k+1}}\min(1,N^k_S(j))  $. 
 
 Montrons alors que l'inégalité $(*)$ est vérifiée pour tout $n \geq 0$. Pour cela, supposons qu'elle soit satisfaite au rang $n $. Appliquons alors le lemme avec
$$b_{n+1}:=(e^{-u_n}(\Id +u_n)-\Id)(a) $$
et $\tau-s=\s_n$. On obtient l'inégalité~:
$$| b_{n+1}|_{s_n-\s_n} \leq  \frac{1 }{\s_n^2} N^1_{s_n}(u_n)^2.$$
En utilisant l'hypothèse de récurrence et la définition de $m$, on obtient l'estimation~:
$$| b_{n+1}|_{s_n-\s_n} \leq \frac{m^2 \s_{n}^{2k+4}}{N^k(j)\s_n^2} \leq   m \s_{n+1}^{2k+2}$$
Comme $j$ est $k$-borné, et $u_{n+1}=j(b_{n+1})$, on en déduit l'inégalité~:
$$N^1_{s_{n+1}}(u_{n+1})\leq \frac{m \s_{n+1}^{2k+2}}{\s_n^k} \leq m \s_{n+1}^{k+2}. $$ 
On a donc bien
$$\sum_{n \geq 0}N^1_s(u_n) \leq \sum_{n \geq 0}N^1_{s_n}(u_n)<3s$$
pour tout $s \leq \tau$. Le théorème est démontré.

 \section{\'Echelonnements  en g\'eom\'etrie analytique}
 \label{S::analytique}
 \subsection{Généralités}
 Soit $\Ft$ un faisceau en espaces vectoriels topologiques définit sur un espace topologique $X$.  On appelle {\em fibre du faisceau $\Ft$ en un compact $K \subset X$}, noté $\Ft_K$, l'espace vectoriel 
 $$\Ft_K=\underrightarrow{\lim}\, \G(U,\Ft) $$
 où $U$ parcourt l'ensemble des ouverts contenant $K$ ordonné par l'inclusion.
   
 Un élément de $\Ft_K$ est une section du faisceau $\Ft$ au voisinage de $K$, pour laquelle on oublie de préciser la taille du voisinage de $K$ sur laquelle
 elle est définie.  Prendre la limite directe revient donc à identifier deux sections qui sont égales sur un ouvert contenant $K$~:
 $$f \sim g \iff \exists U \supset K,\ f_{\mid U}=g_{\mid U}.  $$ 
 Dans le cas où $K$ est réduit à un point, on retrouve la notion de germe en un point. Nous parlerons donc de germes de fonctions holomorphes en un compact. C'est une notion classique (voir par exemple \cite{Cartan_BSMF}). 
 
 En munissant les espaces vectoriels $ \G(U,\Ft)$ de la topologie de la convergence compacte, on munit la limite directe $\Ft_K$ d'une topologie. 
 
 L'espace topologique $\Ft_K$ peut s'obtenir comme limite directe d'espaces de Banach de la façon suivante.
Notons $\Vt$ l'ensemble des voisinages compacts de $K$, ordonné par l'inclusion.
Pour $K' \in \Vt$, on considère  l'espace vectoriel $B(K')$ des fonctions continues sur $K'$ qui sont holomorphes dans l'intérieur de $K'$.  C'est une espace de Banach pour la norme~: 
$$B(K') \to \RM,\ f \mapsto \sup_{z \in K'}| f(z)|.$$
L'espace vectoriel $\Ft_K$  est limite directe des $ B(K')$~:
$$\Ft_K =\underrightarrow{\lim}\, B(K'), \ K' \in \Vt. $$  

Nous allons à présent échelonner ces espaces vectoriels lorsque $K$ est un compact de Stein. D'après le théorème de plongement des espaces de Stein, on peut se limiter au cas des sous-espaces analytiques de 
$\CM^n$~\cite{Bishop_embedding,Narasimhan_embedding,Remmert_embedding,Wiegmann_embedding}.

Soit $K$ un compact de Stein  de  $\CM^n$  et $X$ un sous-espace analytique de $\CM^n$. Posons $K'=K \cap X$. D'après  le théorème $A$ de Cartan, on peut choisir une présentation du module $\Ot_{X,K'}$~:
$$\Ot_{\CM^n,K} \stackrel{C}{\to} \Ot_{\CM^n,K} \to  \Ot_{X,K'} \to 0$$
L'image de cette application est fermée (voir appendice), l'espace vectoriel $\Ot_{X,K'}$ se voit ainsi muni  d'une structure d'espace vectoriel échelonné.

On définit, alors, les échelonnements des fibres en $K'=K \cap X$ pour un faisceau analytique cohérent $\Ft$ sur $X$. Pour cela, on choisit une présentation de
$\Ft_{K'}$~:
$$\Ot_{X,K'}^p \stackrel{C}{\to} \Ot_{X,K'}^q \to  \Ft_{K'} \to 0,$$
et on prend sur $\Ft_{K'}$ la structure échelonnée induite par celle de $\Ot_{\CM^n,K}$. Ainsi chaque structure échelonné
sur $\Ot_{\CM^n,K}$ induit des structures échelonnés sur les fibres en $K'=K \cap X$ d'un faisceau cohérent sur $X$.

 \subsection{La structure échelonnée  $C^0$}
\label{SS::Banach}

Soit $K=(K_s),\  s \in [0,S]$ une famille exhaustive  de compacts  d'un espace analytique $X$.  

  L'espace vectoriel topologique $\Ot_{X,K_0}$ est alors échelonné par les espaces de Banach~:
  $$E_s:=C^0(K_s,\CM) \cap \G(\mathring{K_s},\Ot_X). $$

On induit ainsi une structure échelonné sur la fibre en $K_0$ de tout faisceau cohérent $\Ft$, définit sur un sous-espace analytique de $\CM^n$. Nous l'appellerons la structure $C^0$ associé à la famille $K$, nous la noterons $C^0(K,\Ft)$. On définit ainsi un foncteur de la catégorie des faisceaux cohérents sur $X$ vers celles des espaces vectoriels échelonnés~:
$$C^0(K,-):\ {\rm Coh}(X) \to EVE,\ \Ft \mapsto C^0(K,\Ft).$$ 
   
    \begin{proposition} Soit $(K_s)$ une suite exhaustive de compacts de Stein d'un espace analytique $X$ et $\Ft,\Gt$ des faisceaux analytiques cohérents sur $X$. Tout morphisme 
$$ \Ft_{K_0} \to \Gt_{K_0} $$
de $\Ot_{X,K_0}$-modules définit un morphisme $0$-borné 
$$C^0(K,\Ft) \to C^0(K,\Gt). $$
\end{proposition}   
\begin{proof}
D'après la proposition \ref{P::quotient},  il suffit  de montrer la proposition pour un morphisme de module libres
$$\Ot_{\CM^n,K_0}^p \to \Ot_{\CM^n,K_0}^q $$
La somme de morphismes $0$-borné étant $0$-borné, on peut se restreindre au cas $p=q=1$ ; auquel cas la proposition est évidente, car les espaces
$ C^0(K,\Ot_{\CM^n})_s$ sont des algèbres de Banach.
\end{proof}
  
 \subsection{La structure échelonnée  $L^2$}
Soit $X$ un espace analytique et $K=(K_s)$ une famille exhaustive de compacts de Stein de $X$. 
Les espaces de Hilbert
$$ L^2(K_s,\CM) \cap \G(\mathring{K_s},\Ot_X).$$
munissent  l'espace vectoriel topologique $\Ot_{X,K_0}$ d'un échelonnement.

Ceci définit pour tout faisceau cohérent $\Ft$, une  structure échelonnée sur $\Ft_K$ que nous noterons $L^2(K,\Ft)$. On a, à nouveau, un foncteur
de la catégorie des faisceaux cohérents sur un espace analytique vers celle des espaces vectoriels échelonnés.
 
L'inégalité de Cauchy-Schwarz montre que tout morphisme de module est $0$-borné, mais on a plus~:
\begin{proposition}
\label{P::inverse} Soit $K=(K_s)$ une suite exhaustive de compacts de Stein dans un espace analytique $X$ et $\Ft,\Gt$ des faisceaux analytiques cohérents. Tout morphisme surjectif
$$ \Ft_{K_0} \to \Gt_{K_0} $$
de $\Ot_{X,K_0}$-modules  admet un inverse $0$-borné 
$$L^2(K,\Gt) \to L^2(K,\Ft). $$
\end{proposition}
\begin{proof}
Choisissons $s$ assez petit pour que le morphisme définisse, par restriction, une application linéaire continue 
$$u_s:L^2(K,\Ft)_s \to L^2(K,\Gt)_s.$$
Le noyau de cette application est un sous-espace fermé $F_s$. D'après le théorème de l'image ouverte, la restriction de $u_s$ à l'orthogonal de $F_s$
est un isomorphisme d'espaces de Hilbert. Notons $v_s$ son inverse. On obtient ainsi un inverse à droite linéaire et continu de $u_s$~:
$$\s:L^2(K,\Gt)_s \to L^2(K,\Ft)_s=F_s \oplus F_s^{\perp},\ x \mapsto (0,v_s(x)) .  $$
Il ne reste plus qu'à observer que, pour $s' \leq s$, la multiplication donne des isomorphismes canoniques\footnote{On note $\hat \otimes$ le produit tensoriel topologique projectif voir~\cite{Grothendieck_these,Grothendieck_PTT,Schatten}.}
$$L^2(K,\Ot_X)_{s'}  \hat \otimes L^2(K,\Ft)_s  \approx  L^2(K,\Ft)_{s'} ,\ 
L^2(K,\Ot_X)_{s'}  \hat \otimes L^2(K,\Gt)_s  \approx  L^2(K,\Gt)_{s'} . $$
Via ces isomorphismes, l'application 
$$\id_{s'} \otimes \s:  L^2(K,\Ot_X)_{s'}  \hat \otimes L^2(K,\Gt)_s  \to  L^2(K,\Ot_X)_{s'}  \hat \otimes L^2(K,\Ft)_s  $$
s'identifie à une section continue du morphisme 
$$u_{s'} :L^2(K,\Ft)_{s'} \to L^2(K,\Gt)_{s'} .$$
Comme la norme de $ \id_{s'} \otimes \s$ est égale à celle de $\s$ ceci conclut la démonstration de la
proposition.
\end{proof}

 \subsection{Recouvrements échelonnés}
\label{SS::echelon}
Soit $(K_s) \subset \CM^n,\ s \in [0,S]$ une famille exhaustive de compacts. Munissons $\CM^n$ de coordonnées $(z_1,\dots,z_n)$ et
notons $P$ le polycylindre
$$P=\{ z \in \CM^n : | z_i | \leq 1,\ i=1,\dots,n \}. $$
\begin{definition}
Une famille exhaustive de compact $(K_s)$ est appelée un recouvrement échelonné si pour tout point $z \in K_s$ le polydisque $z+\s P$ est contenu dans le compact  $K_{s+\s}$, pour tous $s \in [0,S[$ et $\s \in [0,S-s[$.
\end{definition}
Dans le cas plus général d'un espace analytique $X$, nous dirons qu'une famille exhaustive de compacts est un {\em recouvrement échelonné}
si on peut plonger $X$ dans $\CM^n$ de telle sorte que cette famille soit obtenue comme intersection d'un recouvrement échelonné de $\CM^n$ avec $X$.

Nous allons construire des recouvrements échelonnés de la façon suivante. Soit
 $$\psi:\CM^n \supset \Omega \to [0,S], S \in [0,+\infty[$$ une fonction propre de classe $C^1$  définie sur un ouvert $\Omega \subset \CM^n$
 dont la dérivée est bornée.  Munissons  $\CM^n$ de la norme $\max_{i =1,\dots,n}| \cdot |$ et soit $\| \cdot \|$ la norme d'opérateur dans $L(\CM^n,\CM)$. Soit  $M$ un majorant de la norme des dérivées de $\psi$~:
 $$\sup_{z \in \Omega} \| D\psi(z)\| \leq M  $$
\begin{lemme} {La famille de compacts $K=(K_s)$ avec $K_s:=\psi^{-1}([0,Ms])$ est un recouvrement échelonné.}
\end{lemme}
\begin{proof}
Soit $z \in K_s$, la formule de Taylor donne~:
$$\psi(z+\s \dt)=\psi(z)+(\int_{t=0}^{t=1} D\psi(z+t\s \dt) dt)\s\dt,\ \dt \in P .$$
On a bien~:
$$  \psi(z+\s\dt)  \leq   \psi(z+\s\dt)  +  \sup_{t \in [0,1]} \| D\psi(z +t\s \dt)\| \s \leq Ms+M\s,$$
ce qui démontre le lemme.
  \end{proof}
 Si un recouvrement $K=(K_s)$ peut-être définit comme dans le lemme, et si de plus la fonction $\psi$ est pluri-sousharmonique, nous dirons que $K$ est {\em un recouvrement Stein-échelonné.}  
\subsection{Comparaisons des échelonnements $C^0$ et $L^2$} 
Soit un faisceau cohérent $\Ft$ sur un espace analytique $X$ et $K$ une  famille exhaustive de compacts .
Comme toute fonction continue sur un compact est intégrable, l'application identité de $\Ot_{X,K_0}$ dans lui-même induit un morphisme $0$-borné
$$C^0(K,\Ft) \to L^2(K,\Ft) $$ 
\begin{proposition} Si $K=(K_s)$ est un recouvrement Stein-échelonné d'un espace analytique $X$, l'identité de $\Ot_{X,K_0}$ induit un morphisme $1$-borné
$$ L^2(K,\Ft)  \to C^0(K,\Ft). $$
\end{proposition}
\begin{proof}
En vertu de la proposition \ref{P::quotient}, il suffit de montrer la proposition pour le faisceau structural des espaces vectoriel $\CM^n$, $n \in \NM$.
Pour $z \in U_s$ et $\s$ fixés, on pose
$$f(z+\s \dt)=\sum_{j \geq 0} a_j \s^j,\ a_j \in \CM^n. $$
Par un calcul direct, on obtient
$$\int_{z+\s P} | f(z)|^2 dV=\sum_{j \geq 0} |a_j|^2 \s^{2j+2} . $$
Comme le recouvrement $K$ est échelonnée, le polycylindre $z+\s P$ est contenu dans $ K_{s+\s}$ donc
$$\int_{z+\s P} | f(z)|^2 dV \leq \int_{K_{s+\s}} | f(z)|^2 dV=| f |_{s+\s}^2  $$
On en déduit les inégalités~:
$$|a_0|= |f(z) | \leq  \s^{-1}\left( \int_{z+\s P} | f(z)|^2 dV \right)^{1/2} \leq  \s^{-1} | f |_{s+\s}.$$
 Ce qui démontre la proposition.  
 \end{proof}
 En combinant cette proposition avec la proposition \ref{P::inverse}, on obtient le
 \begin{corollaire}
 \label{C::inverse} Soit $K=(K_s)$ un recouvrement Stein-échelonné dans un espace analytique $X$ et $\Ft,\Gt$ des faisceaux analytiques cohérents. Tout morphisme surjectif
$$ \Ft_{X,K_0} \to \Gt_{X,K_0} $$
de $\Ot_{X,K_0}$-modules admet un inverse $1$-borné 
$$C^0(K,\Gt) \to C^0(K,\Ft). $$
 \end{corollaire}
 \subsection{Cas réel}
Supposons la variété $X $ munie d'une involution anti-holomorphe 
$$\tau:X \to X $$
et $K$ est un compact contenu dans $X_{\RM }$.  L'espace $\Rt_{X,K }$ des germes de fonctions analytiques réelles le long de $K$ est un sous-espace vectoriel topologique fermé de~$\Ot_{X ,K}$. Il hérite par conséquent des structures échelonnées de $\Ot_{X ,K}$.
 
La partie réelle d'une sous-variété $X \subset \CM^n$ de Stein définie par des fonctions analytiques
 $$g_1,\dots,g_n:X \to \CM ,$$ admet des recouvrements Stein-échelonnés.
Il suffit, en effet, de poser\footnote{Cette fonction m'a été suggérée par P. Dingoyan.}~: 
 $$ \psi=\sum_{i=1}^n | g_i|+\sum_{i=1}^n | z-\tau (z) |$$ 
et de considérer un recouvrement associé à $\psi$ comme dans \ref{SS::echelon}.  Ces recouvrements sont invariants par l'involution $\tau$.
 \subsection{Démonstration des théorèmes \ref{T::det} et  \ref{T::det_bis}}
On muni l'espace vectoriel $\Ot_{X,K}$ d'un échelonnement $C^0$ provenant d'un recouvrement Stein-échelonné.

Par hypothèse, l'application
 $$\Mt_K^2 \otimes_{\Ot_{X,K}} \Der_{\Ot_{X,K}}(I) \to  \Ot_{X,K},\ a \otimes v \mapsto av(f)$$
 contient l'idéal $I^{\nu}$ dans son image. On note $\alg$ la préimage de cet idéal. Les éléments de $\alg$ sont d'ordre
 $2$ et ce module s'identifie canoniquement à un sous-espace d'applications $1$-bornés.
 
 On applique le théorème \ref{T::groupes} avec $a=f$, $E=\Ot_{X,K}$, $M=I^{\nu}$ et $\alg$ comme décrit ci-dessus. D'après le corollaire \ref{C::inverse}, le théorème \ref{T::groupes} s'applique, ce qui démontre le théorème \ref{T::det}.
 
Pour la démonstration du théorème \ref{T::det_bis} on procède de la même manière~: l'application
 $$ \Der_{\Ot_{X,K}}(\Ot_{X,K}) \to  \Ot_{X,K}$$
 contient l'idéal $\Mt_K^{\mu}$ dans son image.  On note $V$ la préimage de cet idéal. Les éléments de $\alg=\Mt_K^2 \otimes V$ sont d'ordre
 $2$ et ce module s'identifie canoniquement à un sous-espace d'applications $1$-bornés. Les inclusions
 $$ \Mt_K^2 \otimes_{\Ot_{X,K}} \Der_{\Ot_{X,K}} (\Ot_{X,K})  \subset  \Mt_K \otimes_{\Ot_{X,K}} \Der_{\Ot_{X,K}} (\Ot_{X,K}) \subset  \Der_{\Ot_{X,K}} (\Mt_K )$$
 montre que $\Mt_K^{\mu+2}$ est stable par $\alg$.
 Le théorème \ref{T::det_bis} est donc également une conséquence du théorème \ref{T::groupes}.
 
\appendix
   \section{Sur l'image d'un morphisme de $\Ot_{X,K}$-modules}
   \subsection{Le théorème de Banach-Köthe}
Toute application linéaire continue entre espaces de Fréchet (localement convexe, métrisable) est ouverte. C'est le {\em théorème de l'image ouverte} appelé aussi {\em théorème de Banach.}
 
 Considérons un espace vectoriel $E$ qui soit l'union dénombrable strictement croissante de sous-espaces de Fréchet~:
 $$E=\bigcup_{n \in \NM} E_n,\ E_{n} \subsetneq E_{n+1}.$$

 Les inclusions $E_n \subset E_{n+1}$ donnent un système direct
 $$0 \to E_0 \to E_1 \to E_2 \to \dots $$
dont $E$ est la limite ; il est donc muni d'une structure d'espace vectoriel topologique. L'espace vectoriel topologique $E$ est réunion dénombrable des $E_i$ qui sont d'intérieur vide, ce n'est donc pas un espace de Baire. En particulier, la topologie de $E$ n'est pas métrisable et le théorème de l'image ouverte ne s'applique pas. On a toutefois le
 \begin{theor}[\cite{Koethe_Banach}] Soit $E$ un espace vectoriel union dénombrable d'espaces de Fréchet. Si $E$ est un espace complet pour la topologie induite par ses sous-espaces de Fréchet alors toute application linéaire continue surjective est ouverte.
 \end{theor}
  
\subsection{\'Enonc\'e du résultat}
Rappelons qu'une application linéaire continue entre espaces vectoriels topologiques $u:E \to F$ est appelée  {\em stricte} si elle induit un isomorphisme
 d'espaces vectoriels topologiques entre $E/\Ker u$ et $\Im u$ \cite{Bourbaki_topologie,Bourbaki_EVT}. 
 \begin{pro}
\label{P::strict}
Soit  $K \subset \CM^n$ un compact et $M,N$ deux  modules de type fini sur l'anneau $\Ot_{\CM^n,K}$. Toute application $\Ot_{\CM^n,K}$-linéaire de $M$ vers $N$ est stricte.  
\end{pro}
C'est un résultat classique~(voir par exemple~\cite{Malgrange_analytic}). Nous allons voir que le théorème de Banach-Köthe permet d'en donner une dé\-monstra\-tion alternative.   
\begin{proof}
Il suffit de démontrer   la proposition pour $M=\Ot_{\CM^n,K}^p$ et $N=\Ot_{\CM^n,K}^q$.  Commençons par le
   \begin{lem}
   \label{L::algebre} Soit $A$ une algèbre topologique. Si pour tout $x \in A$ l'image de la multiplication par $x$ est stricte alors toute application $A$-linéaire $A^n \to A^p$ est également stricte.
 \end{lem}
\begin{proof} 
 Soit $u,v:E \to F$ deux morphismes stricts, je dis qu'alors
\begin{enumerate}[{\rm i)}]
\item $(u,v):E \to F \times F,\ x \mapsto (u(x),v(x))$ est strict ;
\item  $u+v:E \to F,\ x \mapsto u(x)+v(x)$ est strict.
\end{enumerate}
La restriction d'un morphisme strict à un sous-espace vectoriel fermé est à nouveau un morphisme strict.  Donc la
restriction à diagonale $\D$ de l'application
 $$w:E \times E \to F \times F,\ (x,y) \mapsto (u(x),v(y)) $$
 est stricte. Ce qui démontre i). 
 
 Considérons le morphisme strict
 $$s:F \times F \to F,\ (x,y) \mapsto x+y .$$
 La composée de deux morphismes stricts est stricte donc $s \circ w_{| \D}=u+v$ est strict. Ce qui démontre ii). 
 
 Comme toute application linéaire est une somme finie d'applications de rang 1, les affirmations i) et ii) entraînent immédiatement le lemme.
 \end{proof}
D'après le théorème de Banach-Köthe, il nous reste donc à montrer que  la multiplication par $a $ est
d'image fermée, pour tout $a \in \Ot_{\CM^n,K}$. 

Pour cela, considérons deux suites de germes en $K$ de fonctions holomorphes   $(y_k), (x_k)$ avec $y_k=a x_k$. Il s'agit de prouver que si $(y_k)$  est convergente alors $(x_k)$ l'est également. Pour cela, il suffit de démontrer que, pour chaque droite complexe $L \subset \CM^n $, la restriction des $x_k$ à $L$ définit une suite convergente. On est ainsi ramené au cas $n=1$. Si la limite existe elle est unique, il suffit donc de vérifier la propriété localement.
 
 Comme les zéros d'une fonction holomorphe d'une variable sont isolés et comme la propriété d'être holomorphe est locale, on peut  supposer que $K=\{ 0 \} $. Posons alors
$$a(z)=z^db(z),\ b(0) \neq 0,$$
on a 
$$g_k(z)=z^dh_k(z),\ w_k(0) \neq 0.$$ Ce qui montre que la suite $(f_k)$ s'écrit sous la forme 
 $$f_k=\frac{h_k}{b},\ b(0) \neq 0.$$
 La suite $(h_k)$ étant convergente, ceci démontre la convergence de la suite $(f_k)$. La proposition est démontrée.
\end{proof} 
Soit $I \subset  \Ot_{\CM^n,K}$ un idéal engendré par $f_1,\dots,f_k \in  \Ot_{\CM^n,K}$.
En appliquant la proposition à l'application
$$\Ot_{\CM^n,K}^k \to  \Ot_{\CM^n,K},\ (a_1,\dots,a_k) \mapsto \sum_{i=1}^k a_i f_i ,$$
on obtient le
\begin{coro} Tout idéal de l'espace vectoriel topologique $ \Ot_{\CM^n,K}$ définit un sous-espace vectoriel fermé.
\end{coro}

\noindent {\bf Remerciements.}{ Merci à P. Dingoyan et  J. Féjoz pour leur aide, ainsi qu'à J.-C. Yoccoz pour m'avoir signalé une erreur dans la démonstration initiale du théorème \ref{T::produits}.} 
\bibliographystyle{plain-fr}
\bibliography{master}
 \end{document}